\documentclass[12pt,reqno]{amsart}

\usepackage{amssymb,amsmath,graphicx,amsfonts,euscript}
\usepackage{color}

\newtheorem{thm}{Theorem}[section]

\newtheorem{prop}[thm]{Proposition}
\newtheorem{define}[thm]{Definition}
\newtheorem{rem}[thm]{Remark}

\newtheorem{lemma}[thm]{Lemma}

\newcommand{\R}{\mathbb{R}}
\newcommand{\Z}{\mathbb{Z}}

\numberwithin{equation}{section}
\begin{document}
\title[Global well-posedness to Oldroyd-B models]{Global well-posedness to  the subcritical  Oldroyd-B type  models in 2D}

\author[Renhui Wan ]{ Renhui Wan$^{\ast}$}

\address{$^1$ School of Mathematics Sciences,
Zhejiang University,
Hanzhou 310027, China}

\email{rhwanmath@zju.edu.cn, rhwanmath@163.com,21235002@zju.edu.cn}

\vskip .2in
\begin{abstract}
We prove the global well-posedness to the 2D Oldroyd-B type models with $\nu \Lambda^{2\alpha}u$ and $\eta\Lambda^{2\beta}\tau$ satisfying $(i)\  \alpha>1, \eta=0$ or $(ii)\  \alpha=1,\ \beta>0$. By establishing the gradient estimate of $u$, $\tau$ and $L^\infty$ bound of ${\rm curl u+\Lambda^{-2}curldiv \tau}$,   Elgidi-Rousset
(Commun. Pure Appl. Math. online, 2015) obtained the global well-posedness for the case $\nu=0$, $\beta=1$. However, for the cases $(i)$ and $(ii)$, it is difficult to  improve the regularity of $u$ and $\tau$ directly, especially when  $\alpha\rightarrow 1^{+}$ in case $(i)$ and  $\beta\rightarrow 0^{+}$  in case $(ii)$. To overcome this difficulty, we
exploit a new structure of the equations coming from the dissipation and coupled term. Then we prove the global well-posedness to these cases by energy method which brings us closer to the more interesting case $\alpha=1$, $\eta=0$.

\end{abstract}

\maketitle

\vskip .2in
\section{Introduction}
\label{s1}
We are concerned with the  Oldroyd-B type models for visco-elastic flow:
\begin{equation} \label{0.1}
\left\{
\begin{aligned}
& \partial_t u + u\cdot\nabla u -\nu \Delta u +\nabla p = \kappa {\rm div}\tau,  \\
& \partial_t \tau + u\cdot\nabla \tau +\beta_1 \tau-\eta\Delta \tau=Q(\nabla u ,\tau)+\alpha_1 Du,\\
& {\rm div} u=0,\\
& u(0,x) =u_0(x),\ \tau(0,x)=\tau_0(x),
\end{aligned}
\right.
\end{equation}
here $(t,x)\in\mathbb{R}^{+}\times\mathbb{R}^d$, $u,p,\tau$ stand for velocity vector, scalar pressure and symmetric tensor, respectively,  $\nu,\eta,\beta_1$ are nonnegative parameters and $\kappa>0,\  \alpha_1>0$ represent the coupling parameters. $Du,\  W(u)$ are the deformation tensor and vorticity tensor,
$$Du:=\frac{\nabla u+(\nabla u)^\top}{2},\ \ W(u):=\frac{\nabla u-(\nabla u)^\top}{2},$$
respectively, and
$$Q(\nabla u,\tau)=W(u)\tau-\tau W(u)+b(Du \tau+\tau Du),\ b\in [-1,1].$$

For the classical case $\nu>0,\ \eta=0$, Chemin-Masmoudic \cite{CM}  studied the local and global well-posedness of (\ref{0.1}) in $d$ dimensions. But for the global result in $L^p$ framework, they needed the small coupling parameter. This gap was filled in a  recent work by Zi-Fang-Zhang \cite{ZFZ} with the method based on Green's matrix of the linearized system, which was developed in \cite{CMZ2}.  Furthermore, \cite{CM} also established some blowup criteria for local solutions, see \cite{LMZ} for some other related results. For the works on bounded domain and exterior domain, we refer to \cite{BM} and \cite{FHZ} and reference therein.

\vskip .1in
When $\nu>0$, $\eta>0$ and $b=1$, Constantin-Klieg \cite{CK} obtained the global well-posedness for (\ref{0.1}) with density equation in 2D. They derived the global $L^2$ bound of $u$ and  $L^1$ bound of $\tau$, and then improve it to  high regularity by energy method.

\vskip .1in
Very recently, for the case $\nu=0$, $\eta>0$ and $Q(\nabla u,\tau)=0$, Elgidi and  Rousset \cite{ER} proved the global well-posedness of (\ref{0.1}) in 2D by establishing the gradient estimates of $u$ and $\tau$ and exploiting a structure of the equations, i.e., $L^\infty$ bound of curl$u$+$\Lambda^{-2}$curldiv$\tau$. In addition, in Remark 1.3 of that paper, they also expected
similar result held for the generalized version of (\ref{0.1}) with $\nu\Lambda^{2\alpha}u$ and $\eta\Lambda^{2\beta}\tau$, $\alpha+\beta=1,$  called critical case, which is given by:
\begin{equation} \label{1.1}
\left\{
\begin{aligned}
& \partial_t u + u\cdot\nabla u +\nu \Lambda^{2\alpha} u +\nabla p ={\rm div}\tau,  \\
& \partial_t \tau + u\cdot\nabla \tau +\eta\Lambda^{2\beta} \tau= Du,\\
& {\rm div} u=0,\\
& u(0,x) =u_0(x),\ \tau(0,x)=\tau_0(x),
\end{aligned}
\right.
\end{equation}
where we have let $\kappa=\alpha_1=1, \beta_1=0$ without loss of generality  and  the definition of $\Lambda$ can be seen in Section \ref{s2}.

\vskip .1in
However, when $0\le\beta<1,\ \alpha+\beta=1$, one can check that it is difficult to improve the $L^2$ bound of $u$ and $\tau$ to $H^\epsilon$ bound, $\forall\ \epsilon>0$, not to mention gradient estimate, so it seems the method in \cite{ER} can not be applied to the these cases.

\vskip .1in
Motivated by the above argument,  in this paper, we  consider (\ref{1.1}) and get the global well-posedness for two slightly subcritical cases: $(i)\ \nu>0,\alpha>1, \eta=0;$ $(ii)\ \nu>0,\alpha=1,\eta>0,\beta>0$.  The main results can be shown as follows:
\begin{thm}\label{t1}
Consider (\ref{1.1}) with $\nu>0,\alpha>1,\eta=0$ and the initial data $(u_0,\tau_0)\in H^s(\mathbb{R}^2)$, $s>2$, satisfying ${\rm div}u_0=0$ and $\omega_0-\mathcal{R}_\alpha\tau_0\in L^2(\mathbb{R}^2)$. Then there exists a unique global solution $(u,\tau)$ satisfying
$$(u,\tau)\in C([0,T];H^s(\mathbb{R}^2)),\ u\in L^2([0,T]; H^{s+\alpha}(\mathbb{R}^2)), \ \forall\  T>0.$$
Moreover, when $1<\alpha\le \frac{3}{2}$,
\begin{equation}\label{1.2}
\begin{aligned}
\sup_{0\le t\le T}\|(\omega-\mathcal{R}_\alpha\tau)(t)\|_{L^2}^2+&\nu\int_{0}^T\|\Lambda^\alpha(\omega-\mathcal{R}_\alpha\tau)(t)\|_{L^2}^2dt\\
 &\le C(T, \nu, \alpha, \|(u_0,\tau_0)\|_{L^2},\|\omega_0-\mathcal{R}_\alpha\tau_0\|_{L^2}).
\end{aligned}
\end{equation}
\end{thm}

\begin{thm}\label{t2}
Consider (\ref{1.1}) with $\nu>0,\alpha=1,\eta>0,0<\beta\le \frac{1}{2}$ and the initial data $(u_0,\tau_0)\in H^s(\mathbb{R}^2)$, $s>2$, satisfying ${\rm div} u_0=0$.  Then there exists a unique global solution $(u,\tau)$ satisfying
$$(u,\tau)\in C([0,T];H^s(\mathbb{R}^2)), \ u\in L^2([0,T]; H^{s+1}(\mathbb{R}^2)), \tau\in L^2([0,T]; H^{s+\beta}(\mathbb{R}^2)),\ \forall \  T>0.$$
Moreover,
\begin{equation}\label{1.3}
\begin{aligned}
\sup_{0\le t\le T}\|(\omega-\mathcal{R}_1\tau)(t)\|_{L^2}^2+&\nu\int_{0}^T\|\nabla(\omega-\mathcal{R}_1\tau)(t)\|_{L^2}^2dt\\
 &\le C(T, \nu, \eta, \beta, \|(u_0,\tau_0)\|_{L^2},\|\omega_0-\mathcal{R}_1\tau_0\|_{L^2}).
\end{aligned}
\end{equation}
\end{thm}
In the  Theorem \ref{t1} and \ref{t2},
$$\omega=\rm curl\  u,\ \mathcal{R}_\alpha \tau:=\frac{1}{\nu}\Lambda^{-2\alpha}curldiv\  \tau, \ \alpha\ge 1, $$
and $\omega_0$ is the initial data of $\omega$.

\begin{rem}\label{r1}
The proofs of the above Theorems do not appear to be able to the case $\nu>0,\alpha=1,\eta=0$ unless we  neglect   the coupling term ${\rm div}\tau$ or
$Du$ in (\ref{1.1}). In fact, provided without ${\rm div}\tau$, we can first solve the first equation, which is actually 2D Navier-Stokes system (global well-posedness is well-known for smooth initial data), and then work out the second linear equation. On the other hand, when neglecting $Du$, we can obtain the global well-posedness by following the idea in Section \ref{s3} and using the $L^\infty$ norm of $\tau$. Additionally, the second special case is very similar to 2D Boussinesq equations which have attracted  many mathematicians' attention recently, (see, e.g., \cite{Hmidi1,Hmidi2,JMWZ,SW}).
\end{rem}

\begin{rem}\label{r2}
The main result in Theorem \ref{t2} also holds for the case $\nu>0,\alpha=1,\beta>\frac{1}{2}$. As a matter of fact, with a similar argument as the proof of the case $\alpha>\frac{3}{2}$ in Theorem \ref{t1}, we can get the global well-posedness of (\ref{1.1}) for the cases $\alpha\ge1,$ $\beta>\frac{1}{2}$, whose proof is omitted.
\end{rem}

Now, let us sketch the difficulty of this problem and our idea. The most difficult situations are $\alpha\rightarrow 1^{+}$ in case ($i$)  and $\beta\rightarrow0^{+}$ in case ($ii$). We require $\alpha\ge 1$ to improve the regularity of $u$, and then need $\alpha>\frac{3}{2}$ or $\beta>\frac{1}{2}$ to improve the regularity of $\tau$. But the conditions in our main results are weaker than these.

\vskip .1in
Our proof is exploiting the structure of (\ref{1.1}) and prove the crucial  estimate (\ref{1.2}) and (\ref{1.3}) by establishing and applying some commutator estimates, respectively. Then we overcome the difficulty and complete the proof  by applying   regularity criteria in inhomogeneous Besov space with non-positive index.
\vskip .1in

The present paper is structured as follows:\\
In the next section, we provide the definition of Besov spaces and related facts such as Bernstein's inequality, and then
prove some important commutator estimates. In the third section, we prove Theorem \ref{t1}, while Theorem \ref{t2} is proved in the following section.
Finally, in the Appendix, we present the regularity criteria for the general cases of (\ref{1.1}), and then present the proof of the inequality (\ref{3.12}).

\vskip .1in
Let us complete this section by describing the notations we shall use in this paper.\\
{\bf Notations} For $A$, $B$ two operator, we denote $[A,B]=AB-BA$, the commutator between $A$ and $B$. The uniform constant  $C$ is  different on different lines,  while the constant $C(\cdot)$ means a constant depends on the element(s) in bracket. $(c_j)_{j\in\Z}$ will be a generic element of $l^2(\Z)$ so that $\sum_{j\in \Z}c_j^2=1$.  We also use $L^p$, $\dot{H}^s$ ($H^s$) and $\dot{B}_{p,r}^s$ ($B_{p,r}^s$) to stand for  $L^p(\mathbb{R}^d)$, $\dot{H}^s(\mathbb{R}^d)$ ($H^s(\mathbb{R}^d)$) and $\dot{B}_{p,r}^s(\mathbb{R}^d)$ ($B_{p,r}^s(\mathbb{R}^d)$), respectively.   We shall denote by $(a|b)$ the $L^2$ inner product
of $a$ and $b$. ${\bf 1}$ stands for the characteristic function.

\vskip .4in
\section{ Preliminaries}
\label{s2}
In this section, we give some necessary definitions,  proposition and lemmas.
\vskip .1in
The fractional Laplacian operator $\Lambda^\alpha=(-\Delta)^\frac{\alpha}{2}$ $(\alpha\ge0)$ is defined through the Fourier transform, namely,
$$\widehat{\Lambda^\alpha f}(\xi)=|\xi|^\alpha \widehat{f}(\xi),$$
where the Fourier transform is given by
$$\widehat{f}(\xi)=\int_{\mathbb{R}^d}e^{-ix\cdot\xi}f(x)dx.$$
Let $\mathfrak{B}=\{\xi\in\mathbb{R}^d,\ |\xi|\le\frac{4}{3}\}$ and $\mathfrak{C}=\{\xi\in\mathbb{R}^d,\ \frac{3}{4}\le|\xi|\le\frac{8}{3}\}$. Choose two nonnegative smooth radial function $\chi,\ \varphi$ supported, respectively, in $\mathfrak{B}$ and $\mathfrak{C}$ such that
$$\chi(\xi)+\sum_{j\ge0}\varphi(2^{-j}\xi)=1,\ \ \xi\in\mathbb{R}^d,$$
$$\sum_{j\in\mathbb{Z}}\varphi(2^{-j}\xi)=1,\ \ \xi\in\mathbb{R}^d\setminus\{0\}.$$
We denote $\varphi_{j}=\varphi(2^{-j}\xi),$ $h=\mathfrak{F}^{-1}\varphi$ and $\tilde{h}=\mathfrak{F}^{-1}\chi,$ where $\mathfrak{F}^{-1}$ stands for the inverse Fourier transform. Then the dyadic blocks
$\Delta_{j}$, $\dot{\Delta}_j$, $S_j$ and $\dot{S}_{j}$ can be defined as follows
\begin{equation*}
\left\{
\begin{aligned}
& \Delta_{j}f=\varphi(2^{-j}D)f=2^{jd}\int_{\mathbb{R}^d}h(2^jy)f(x-y)dy, \ j\ge0,  \\
& \Delta_{-1}f=\chi(D)f=\mathfrak{F}^{-1}(\chi(\xi)\widehat{f}(\xi)), \ \Delta_j f=0,\ j\le -2, \\
&  S_j f=\sum_{k\le j-1}\Delta_k f=\chi(2^{-j}D)f=2^{jd}\int_{\mathbb{R}^d}\tilde{h}(2^jy)f(x-y)dy,\ j\ge0,\\
&  S_j f=0,\ j\le -1,
\end{aligned}
\right.
\end{equation*}

\begin{equation*}
\left\{
\begin{aligned}
& \dot{\Delta}_{j}f=\varphi(2^{-j}D)f=2^{jd}\int_{\mathbb{R}^d}h(2^jy)f(x-y)dy, \ j\in\mathbb{Z},  \\
&  \dot{S}_j f=\chi(2^{-j}D)f=2^{jd}\int_{\mathbb{R}^d}\tilde{h}(2^jy)f(x-y)dy,\ j\in\mathbb{Z}.
\end{aligned}
\right.
\end{equation*}
We   easily verifies that with our choice of $\varphi$ and $\chi$,
$$\Delta_{j}\Delta_{k}f=0\ {\rm if} \ |j-k|\ge2\ \ {\rm and}\ \  \Delta_{j}(S_{k-1}f\Delta_{k}f)=0\  {\rm if}\  |j-k|\ge5.$$
$$\dot{\Delta}_{j}\dot{\Delta}_{k}f=0\ {\rm if} \ |j-k|\ge2\ \ {\rm and}\ \  \dot{\Delta}_{j}(\dot{S}_{k-1}f\dot{\Delta}_{k}f)=0\  {\rm if}\  |j-k|\ge5.$$
Let us recall the definition of  the homogeneous and inhomogeneous Besov spaces.
\begin{define}\label{HB}
 Let $s\in \mathbb{R}$, $(p,q)\in[1,\infty]^2,$ the homogeneous Besov space $\dot{B}_{p,q}^s(\R^d)$ is defined by
$$\dot{B}_{p,q}^{s}(\mathbb{R}^d)=\{f\in \mathfrak{S}'(\R^d);\ \|f\|_{\dot{B}_{p,q}^{s}(\R^d)}<\infty\},$$
where
\begin{equation*}
\|f\|_{\dot{B}_{p,q}^s(\R^d)}=\left\{\begin{aligned}
&\displaystyle (\sum_{j\in \mathbb{Z}}2^{sqj}\|\dot{\Delta}_{j}f\|_{L^p(\R^d)}^{q})^\frac{1}{q},\ \ \ \ {\rm for} \ \ 1\le q<\infty,\\
&\displaystyle \sup_{j\in\mathbb{Z}}2^{sj}\|\dot{\Delta}_{j}f\|_{L^p(\R^d)},\ \ \ \ \ \ \ \ {\rm for}\ \ q=\infty,\\
\end{aligned}
\right.
\end{equation*}
and $\mathfrak{S}'(\mathbb{R}^d)$ denotes the dual space of $\mathfrak{S}(\R^d)=\{f\in\mathcal{S}(\mathbb{R}^d);\ \partial^{\alpha}\hat{f}(0)=0;\ \forall\ \alpha\in \ \mathbb{N}^d $\ {\rm multi-index}\} and can be identified by the quotient space of $\mathcal{S'}/\mathcal{P}$ with the polynomials space $\mathcal{P}$.
\end{define}

\begin{define}\label{INB}
 Let $s\in \mathbb{R}$ and $(p,q)\in [1,\infty]^2$, the inhomogeneous Besov space $B_{p,q}^s(\mathbb{R}^d)$ is defined by
$${B}_{p,q}^{s}(\mathbb{R}^d)=\{f\in {S}'(\R^d);\ \|f\|_{{B}_{p,q}^{s}(\R^d)}<\infty\},$$
where
\begin{equation*}
\|f\|_{{B}_{p,q}^s(\R^d)}=\left\{\begin{aligned}
&\displaystyle (\sum_{j\in \mathbb{Z}}2^{sqj}\|\Delta_{j}f\|_{L^p(\R^d)}^{q})^\frac{1}{q},\ \ \ \ {\rm for} \ \ 1\le q<\infty,\\
&\displaystyle \sup_{j\in\mathbb{Z}}2^{sj}\|\Delta_{j}f\|_{L^p(\R^d)},\ \ \ \ \ \ \ \ {\rm for}\ \ q=\infty.\\
\end{aligned}
\right.
\end{equation*}
\end{define}

For the special case $p=q=2$, we have
$$\|f\|_{\dot{H}^s(\R^d)} \thickapprox \|f\|_{\dot{B}_{2,2}^s(\R^d)},\ \|f\|_{{H}^s(\R^d)} \thickapprox \|f\|_{{B}_{2,2}^s(\R^d)},$$
and if $s<0$, one can check that
\begin{equation}\label{E}
\|f\|_{\dot{B}_{p,q}^s(\R^d)}\thickapprox \left\|2^{js}\|\dot{S}_j f\|_{L^p(\R^d)}\right\|_{l^q(\Z)},\ \|f\|_{{B}_{p,q}^s(\R^d)} \thickapprox\left\|2^{js}\|S_j f\|_{L^p(\R^d)}\right\|_{l^q(\Z)},
\end{equation}
where $a\thickapprox b$ means $C^{-1}b\le a\le C b$ for some positive  constant $C$.
\vskip .1in
The $\dot{H}^s(\R^d)$  and $H^s(\R^d)$ ($s>0$) norm of $f$ can be also defined as follows:
$$\|f\|_{\dot{H}^s(\R^d)}:{=} \|\Lambda^s f\|_{L^2(\R^d)}$$
and
$$\|f\|_{H^s(\R^d)}:{=} \| f\|_{L^2(\R^d)}+\|\Lambda^s f\|_{L^2(\R^d)}\thickapprox \| f\|_{L^2(\R^d)}+\|f\|_{\tilde{H}^s(\R^d)},$$
where
$$\|f\|_{\tilde{H}^s(\R^d)}^2:=\sum_{j\ge 0}2^{2js}\|\Delta_j f\|_{L^2(\R^d)}^2.$$

The following proposition and lemmas
provide Bernstein type inequalities for fractional derivatives and commutator estimates.
\begin{prop}
Let $\gamma\ge0$. Let $1\le p\le q\le \infty$.
\begin{enumerate}
\item[1)] If $f$ satisfies
$$
\mbox{supp}\, \widehat{f} \subset \{\xi\in \mathbb{R}^d: \,\, |\xi|
\le \mathcal{K} 2^j \},
$$
for some integer $j$ and a constant $\mathcal{K}>0$, then
$$
\|(-\Delta)^\gamma f\|_{L^q(\mathbb{R}^d)} \le C_1(\gamma,p,q)\, 2^{2\gamma j +
j d(\frac{1}{p}-\frac{1}{q})} \|f\|_{L^p(\mathbb{R}^d)}.
$$
\item[2)] If $f$ satisfies
\begin{equation*}\label{spp}
\mbox{supp}\, \widehat{f} \subset \{\xi\in \mathbb{R}^d: \,\, \mathcal{K}_12^j
\le |\xi| \le \mathcal{K}_2 2^j \}
\end{equation*}
for some integer $j$ and constants $0<\mathcal{K}_1\le \mathcal{K}_2$, then
$$
C_1(\gamma,p,q)\, 2^{2\gamma j} \|f\|_{L^q(\mathbb{R}^d)} \le \|(-\Delta)^\gamma
f\|_{L^q(\mathbb{R}^d)} \le C_2(\gamma,p,q)\, 2^{2\gamma j +
j d(\frac{1}{p}-\frac{1}{q})} \|f\|_{L^p(\mathbb{R}^d)}.
$$
\end{enumerate}
\end{prop}

\begin{lemma} \cite{BCD}
Let $\theta$ be a $C^1$ function on $\R^d$ such that $(1+|\cdot|)\widehat{\theta}\in L^1(\R^d).$ There exists a constant $C$ such that for any Lipschitz function
$a$ with gradient in $L^p(\R^d)$ and any function $b$ in $L^q(\R^d)$, we have for any positive $\lambda$,
\begin{equation}\label{SCE}
\|[\theta(\lambda^{-1}D),a]b\|_{L^r(\R^d)}\le C\lambda^{-1}\|\nabla a\|_{L^p(\R^d)}\|b\|_{L^q(\R^d)}, \ \ {\rm with}\ \frac{1}{p}+\frac{1}{q}=\frac{1}{r}.
\end{equation}
\end{lemma}

\begin{lemma}\cite{Wan,WuJMFM}
Let $1\le p,p_1,p_2\le \infty$ satisfying $1+\frac{1}{p}=\frac{1}{p_1}+\frac{1}{p_2}.$ If $xh\in L^{p_1}(\R^d)$, $\nabla f\in L^\infty (\R^d)$ and
$g\in L^{p_2}(\R^d)$, then
\begin{equation}\label{WuJMFM}
\|h\star(fg)-f(h\star g)\|_{L^{p}(\R^d)}\le C\|xh\|_{L^{p_1}(\R^d)}\|\nabla f\|_{L^\infty(\R^d)}\|g\|_{L^{p_2}(\R^d)},
\end{equation}
where $C$ is a constant independent of $f,g,h$.
\end{lemma}

For more details about Besov space and Sobolev space such as some useful embedding inequalities, we refer to
\cite{BCD}, \cite{Grafakos} and \cite{Stein}.\\

In the following proof, we shall frequently use homogeneous and inhomogeneous Bony's decomposition. For the homogeneous Bony's decomposition,
if $u,v\in \mathfrak{S}'(\R^d)$,
$$uv=\sum_{j\in \mathbb{Z}}\dot{S}_{j-1}u\dot{\Delta}_{j}v+\sum_{j\in \mathbb{Z}}\dot{\Delta}_{j}u\dot{S}_{j-1}v+\sum_{j\in\mathbb{Z}}\dot{\Delta}_{j}u\tilde{\dot{\Delta}}_{j}v,$$
where $\tilde{\dot{\Delta}}_{j}=\dot{\Delta}_{j-1}+\dot{\Delta}_{j}+\dot{\Delta}_{j+1},$ which is   frequently applied to split the commutator $\Upsilon_1=[\dot{\Delta}_{j}, u]v$, namely,
\begin{equation*}
\begin{aligned}
\Upsilon_1=&\sum_{|k-j|\le 4}[\dot{\Delta}_{j},\dot{S}_{k-1}u]\dot{\Delta}_{k}v+\sum_{|k-j|\le 4}\dot{\Delta}_{j}(\dot{\Delta}_{k}u\ \dot{S}_{k-1}v)\\
&+\sum_{k\ge j-2}\dot{\Delta}_{k}u\ \dot{\Delta}_{j}\dot{S}_{k+2}v+\sum_{k\ge j-3}\dot{\Delta}_{j}(\dot{\Delta}_{k}u\ \tilde{\dot{\Delta}}_{k}v).
\end{aligned}
\end{equation*}
Similarly, for the inhomogeneous Bony's decomposition, if $u,v\in \mathcal{S}'(\R^d)$
$$uv=\sum_{j\in \mathbb{Z}}S_{j-1}u\Delta_{j}v+\sum_{j\in \mathbb{Z}}\Delta_{j}u S_{j-1}v+\sum_{j\in\mathbb{Z}}\Delta_{j}u\tilde{\Delta}_{j}v,$$
where $\tilde{\Delta}_{j}=\Delta_{j-1}+\Delta_{j}+\Delta_{j+1},$ which is also  frequently applied to split the commutator $\Upsilon_2=[\Delta_{j}, u]v$, namely,
\begin{equation*}
\begin{aligned}
\Upsilon_2=&\sum_{|k-j|\le 4}[\Delta_{j},S_{k-1}u]\Delta_{k}v+\sum_{|k-j|\le 4}\Delta_{j}(\Delta_{k}u\ S_{k-1}v)\\
&+\sum_{k\ge j-2}\Delta_{k}u\ \Delta_{j}S_{k+2}v+\sum_{k\ge j-3}\Delta_{j}(\Delta_{k}u\ \tilde{\Delta}_{k}v).
\end{aligned}
\end{equation*}

\vskip .1in

The rest of this section is devoted to the proof of some new commutator estimates.
\begin{lemma}\label{lemma1}
($a$) Let $\alpha>0$ and  $s>0$, there holds
\begin{equation}\label{2.1}
\sum_{j\ge0}2^{2js}([\Delta_j,u\cdot\nabla]u|\Delta_j u)\le C\|\nabla u\|_{B_{\infty,\infty}^{-\alpha}}\|u\|_{H^s}\|u\|_{H^{s+\alpha}}.
\end{equation}
($b$) Let $\alpha>0,$ $\beta>0$ and $s>0$, there holds
\begin{equation}\label{2.2}
\begin{aligned}
\sum_{j\ge0}2^{2js}&([\Delta_j,u\cdot\nabla]\tau|\Delta_j \tau)\\
\le& C\left\{\|\nabla u\|_{B_{\infty,\infty}^{-\beta}}\|\tau\|_{H^s}\|\tau\|_{H^{s+\beta}}+
\|\nabla \tau\|_{B_{\infty,\infty}^{-\alpha}}\|\tau\|_{H^s}\|u\|_{H^{s+\alpha}}
\right\}.
\end{aligned}
\end{equation}
($c$) Let $s_2>0$ and $s_2-s_1-\alpha<0$, there holds
\begin{equation}\label{2.3}
\begin{aligned}
\sum_{j\ge0}2^{2js}&([\Delta_j,u\cdot\nabla]\tau|\Delta_j \tau)\\
\le& C\left\{\|\nabla u\|_{L^\infty}\|\tau\|_{H^{s_2}}^2+
\|\nabla \tau\|_{B_{\infty,\infty}^{s_2-s_1-\alpha}}\|\tau\|_{H^{s_2}}\|u\|_{H^{s_1+\alpha}}
\right\}.
\end{aligned}
\end{equation}
\end{lemma}
\begin{proof}[Proof of Lemma \ref{lemma1}]
($a$) By inhomogeneous Bony's decomposition,
\begin{equation}\label{2.4}
\begin{aligned}
|([\Delta_j,&u\cdot\nabla]u|\Delta_j u)|\le \|[\Delta_j, u\cdot\nabla]u\|_{L^2}\|\Delta_j u\|_{L^2}\\
\le& \|\Delta_j u\|_{L^2}\{\sum_{|k-j|\le4}\|[\Delta_j,S_{k-1}u\cdot\nabla]\Delta_k u\|_{L^2}+
\sum_{|k-j|\le4}\|\Delta_j(\Delta_k u\cdot\nabla S_{k-1}u)\|_{L^2}\\
&+\sum_{k\ge j-2}\|\Delta_k u\cdot\nabla \Delta_j S_{k+2}u\|_{L^2}+
\sum_{k\ge j-3}\|\Delta_j(\Delta_k u\cdot\nabla \tilde{\Delta}_k u)\|_{L^2}\}\\
:=& I_{11}+I_{12}+I_{13}+I_{14}.
\end{aligned}
\end{equation}
By H\"{o}lder's inequality, (\ref{SCE}) and (\ref{E}),
\begin{equation*}
\begin{aligned}
|I_{11}|\le& C\|\Delta_j u\|_{L^2}\sum_{|k-j|\le 4}\|\nabla S_{k-1}u\|_{L^\infty}\|\Delta_k u\|_{L^2}\\
\le& C\|\nabla u\|_{B_{\infty,\infty}^{-\alpha}}\|\Delta_j u\|_{L^2}\sum_{|k-j|\le4}2^{k\alpha}\|\Delta_k u\|_{L^2}\\
\le& C2^{-2js}\|\nabla u\|_{B_{\infty,\infty}^{-\alpha}}2^{js}\|\Delta_j u\|_{L^2}\sum_{|k-j|\le4}2^{(j-k)s}2^{k(s+\alpha)}\|\Delta_k u\|_{L^2}\\
\le& C 2^{-2js}c_j\|\nabla u\|_{B_{\infty,\infty}^{-\alpha}}\|u\|_{H^s}\|u\|_{H^{s+\alpha}}\sum_{|k-j|\le4}2^{(j-k)s}c_k\\
\le& C  2^{-2js}c_j^2\|\nabla u\|_{B_{\infty,\infty}^{-\alpha}}\|u\|_{H^s}\|u\|_{H^{s+\alpha}}.
\end{aligned}
\end{equation*}
Repeating the estimate of $I_{11}$,
$$|I_{12}|\le C 2^{-2js}c_j^2\|\nabla u\|_{B_{\infty,\infty}^{-\alpha}}\|u\|_{H^s}\|u\|_{H^{s+\alpha}}.$$
Similarly,  by H\"{o}lder's inequality,
\begin{equation*}
\begin{aligned}
|I_{13}|\le& C\|\nabla \Delta_j u\|_{L^\infty}\|\Delta_j u\|_{L^2}\sum_{k\ge j-2}\|\Delta_k u\|_{L^2}\\
\le& C2^{j\alpha}\|\nabla u\|_{B_{\infty,\infty}^{-\alpha}}\|\Delta_j u\|_{L^2}\sum_{k\ge j-2}\|\Delta_k u\|_{L^2}\\
\le& C2^{-2js}\|\nabla u\|_{B_{\infty,\infty}^{-\alpha}}2^{js}\|\Delta_j u\|_{L^2}\sum_{k\ge j-2}2^{(j-k)(s+\alpha)}2^{k(s+\alpha)}\|\Delta_k u\|_{L^2}\\
\le& C2^{-2js}c_j\|\nabla u\|_{B_{\infty,\infty}^{-\alpha}}\|u\|_{H^s}\|u\|_{H^{s+\alpha}}\sum_{k\ge j-2}2^{(j-k)(s+\alpha)}c_k\\
\le& C2^{-2js}c_j^2\|\nabla u\|_{B_{\infty,\infty}^{-\alpha}}\|u\|_{H^s}\|u\|_{H^{s+\alpha}}\ (s+\alpha>0)
\end{aligned}
\end{equation*}
and
\begin{equation*}
\begin{aligned}
|I_{14}|\le& C\|\Delta_j u\|_{L^2}\sum_{k\ge j-3}\|\nabla \tilde{\Delta}_k u\|_{L^\infty}\|\Delta_k u\|_{L^2}\\
\le& C \|\nabla u\|_{B_{\infty,\infty}^{-\alpha}}\|\Delta_j u\|_{L^2}\sum_{k\ge j-3}2^{k\alpha}\|\Delta_k u\|_{L^2}\\
\le& C 2^{-2js}\|\nabla u\|_{B_{\infty,\infty}^{-\alpha}}2^{js}\|\Delta_j u\|_{L^2}\sum_{k\ge j-3}2^{(j-k)s}2^{k(s+\alpha)}\|\Delta_k u\|_{L^2}\\
\le& C 2^{-2js}c_j\|\nabla u\|_{B_{\infty,\infty}^{-\alpha}}\|u\|_{H^s}\|u\|_{H^{s+\alpha}}\sum_{k\ge j-3}2^{(j-k)s}c_k,\\
\le& C 2^{-2js}c_j^2\|\nabla u\|_{B_{\infty,\infty}^{-\alpha}}\|u\|_{H^s}\|u\|_{H^{s+\alpha}} \ (s>0),
\end{aligned}
\end{equation*}
where we have used the following Young's inequalities for series for the estimates of $I_{13}$ and $I_{14}$, respectively,
$$\|\sum_{k\ge j-2}2^{(j-k)(s+\alpha)}c_k\|_{l^2(\Z)}\le C\|2^{k(s+\alpha)}\mathbf{1}_{k\le 2}\|_{l^1(\Z)}\|c_k\|_{l^2(\Z)}<\infty,\ s+\alpha>0,$$
$$\|\sum_{k\ge j-3}2^{(j-k)s}c_k\|_{l^2(\Z)}\le C\|2^{ks}\mathbf{1}_{k\le 3}\|_{l^1(\Z)}\|c_k\|_{l^2(\Z)}<\infty,\ s>0. $$
Collecting the estimates of $I_{1i}$, $i=1,2,3,4$ in (\ref{2.4}) yields (\ref{2.1}).\\
($b$) For the estimate of (\ref{2.2}), we use a similar procedure as the proof of (\ref{2.1}).
By inhomogeneous Bony's decomposition again,
\begin{equation}\label{2.5}
\begin{aligned}
|([\Delta_j,&u\cdot\nabla]\tau|\Delta_j \tau)|\le \|[\Delta_j, u\cdot\nabla]\tau\|_{L^2}\|\Delta_j \tau\|_{L^2}\\
\le& \|\Delta_j \tau\|_{L^2}\{\sum_{|k-j|\le4}\|[\Delta_j,S_{k-1}u\cdot\nabla]\Delta_k \tau\|_{L^2}+
\sum_{|k-j|\le4}\|\Delta_j(\Delta_k u\cdot\nabla S_{k-1}\tau)\|_{L^2}\\
&+\sum_{k\ge j-2}\|\Delta_k u\cdot\nabla \Delta_j S_{k+2}\tau\|_{L^2}+
\sum_{k\ge j-3}\|\Delta_j(\Delta_k u\cdot\nabla \tilde{\Delta}_k \tau)\|_{L^2}\}\\
:=& I_{21}+I_{22}+I_{23}+I_{24}.
\end{aligned}
\end{equation}
As the estimate of $I_{1i}$, $i=1,2,3,4$, $I_{2i}$  can be bounded as follows:
$$|I_{21}|\le C2^{-2js}c_j^2\|\nabla u\|_{B_{\infty,\infty}^{-\beta}}\|\tau\|_{H^s}\|\tau\|_{H^{s+\beta}},$$
$$|I_{22}|\le C2^{-2js}c_j^2\|\nabla \tau\|_{B_{\infty,\infty}^{-\alpha}}\|\tau\|_{H^s}\|u\|_{H^{s+\alpha}},$$
$$|I_{23}|\le C2^{-2js}c_j^2\|\nabla \tau\|_{B_{\infty,\infty}^{-\alpha}}\|\tau\|_{H^s}\|u\|_{H^{s+\alpha}},\  (s+\alpha>0),$$
$$|I_{24}|\le C2^{-2js}c_j^2\|\nabla \tau\|_{B_{\infty,\infty}^{-\alpha}}\|\tau\|_{H^s}\|u\|_{H^{s+\alpha}},\  (s>0).$$
(\ref{2.2}) an be derived  from inserting the above four estimates into (\ref{2.5}).\\
(c) We suffices to  give the new bound of $I_{2i}$, $i=1,2,3,4$.  By H\"{o}lder's inequality and (\ref{SCE}), we have
\begin{equation*}
\begin{aligned}
|I_{21}|\le& C\|\Delta_j \tau\|_{L^2}\sum_{|k-j|\le4}\|\nabla S_{k-1}u\|_{L^\infty}\|\Delta_k \tau\|_{L^2}\\
\le& C2^{-2js_2}\|\nabla u\|_{L^\infty}2^{js_2}\|\Delta_j \tau\|_{L^2}\sum_{|k-j|\le 4}2^{(j-k)s_2}2^{ks_2}\|\Delta_k \tau\|_{L^2}\\
\le& C2^{-2js_2}c_j\|\nabla u\|_{L^\infty}\|\tau\|_{H^{s_2}}^2\sum_{|k-j|\le 4}2^{(j-k)s_2}c_k\\
\le& C2^{-2js_2}c_j^2\|\nabla u\|_{L^\infty}\|\tau\|_{H^{s_2}}^2.
\end{aligned}
\end{equation*}
By H\"{o}lder's inequality  and (\ref{E}),
\begin{equation*}
\begin{aligned}
|I_{22}|\le& C\|\Delta_j \tau\|_{L^2}\sum_{|k-j|\le4}\|\nabla S_{k-1}\tau\|_{L^\infty}\|\Delta_k u\|_{L^2}\\
\le& C\|\nabla \tau\|_{B_{\infty,\infty}^{s_2-s_1-\alpha}}\|\Delta_j \tau\|_{L^2}\sum_{|k-j|\le4}2^{k(\alpha+s_1-s_2)}\|\Delta_k u\|_{L^2}\\
\le& C2^{-2js_2}\|\nabla \tau\|_{B_{\infty,\infty}^{s_2-s_1-\alpha}}2^{js_2}\|\Delta_j \tau\|_{L^2}\sum_{|k-j|\le4}2^{(j-k)s_2}2^{k(\alpha+s_1)}\|\Delta_k u\|_{L^2}\\
\le& C2^{-2js_2}c_j\|\nabla \tau\|_{B_{\infty,\infty}^{s_2-s_1-\alpha}}\|\tau\|_{H^{s_2}}\|u\|_{H^{s_1+\alpha}}\sum_{|k-j|\le4}2^{(j-k)s_2}c_k\\
\le& C2^{-2js_2}c_j^2\|\nabla \tau\|_{B_{\infty,\infty}^{s_2-s_1-\alpha}}\|\tau\|_{H^{s_2}}\|u\|_{H^{s_1+\alpha}}.
\end{aligned}
\end{equation*}
By H\"{o}lder's inequality and Young's inequality for series,
\begin{equation*}
\begin{aligned}
|I_{23}|\le& C\|\Delta_j \tau\|_{L^2}\|\nabla\Delta_j  \tau\|_{L^\infty}\sum_{k\ge j-2}\|\Delta_k u\|_{L^2}\\
\le& C\|\nabla \tau\|_{B_{\infty,\infty}^{s_2-s_1-\alpha}}2^{j(\alpha+s_1-s_2)}\|\Delta_j \tau\|_{L^2}\sum_{k\ge j-2}\|\Delta_k u\|_{L^2}\\
\le& C2^{-2js_2}\|\nabla \tau\|_{B_{\infty,\infty}^{s_2-s_1-\alpha}}2^{js_2}\|\Delta_j \tau\|_{L^2}\sum_{k\ge j-2}2^{(j-k)(s_1+\alpha)}2^{k(s_1+\alpha)}\|\Delta_k u\|_{L^2}\\
\le& C2^{-2js_2}c_j\|\nabla \tau\|_{B_{\infty,\infty}^{s_2-s_1-\alpha}}\|\tau\|_{H^{s_2}}\|u\|_{H^{s_1+\alpha}}\sum_{k\ge j-2}2^{(j-k)(s_1+\alpha)}c_k\\
\le& C2^{-2js_2}c_j^2\|\nabla \tau\|_{B_{\infty,\infty}^{s_2-s_1-\alpha}}\|\tau\|_{H^{s_2}}\|u\|_{H^{s_1+\alpha}} \ (s_1+\alpha>0),
\end{aligned}
\end{equation*}

\begin{equation*}
\begin{aligned}
|I_{24}|\le& C\|\Delta_j \tau\|_{L^2}\sum_{k\ge j-3}\|\Delta_k u\|_{L^2}\|\tilde{\Delta}_k \nabla \tau\|_{L^\infty}\\
\le& C\|\nabla \tau\|_{B_{\infty,\infty}^{s_2-s_1-\alpha}}\|\Delta_j \tau\|_{L^2}\sum_{k\ge j-3}2^{k(\alpha+s_1-s_2)}\|\Delta_k u \|_{L^2}\\
\le& C2^{-2js_2}\|\nabla \tau\|_{B_{\infty,\infty}^{s_2-s_1-\alpha}}2^{js_2}\|\Delta_j \tau\|_{L^2}\sum_{k\ge j-3}2^{(j-k)s_2}2^{k(\alpha+s_2)}\|\Delta_k u\|_{L^2}\\
\le& C2^{-2js_2} c_j^2\|\nabla \tau\|_{B_{\infty,\infty}^{s_2-s_1-\alpha}}\|\tau\|_{H^{s_2}}\|u\|_{H^{s_1+\alpha}}\ (s_2>0).
\end{aligned}
\end{equation*}
Combining with the above new estimates in (\ref{2.5}) yields the desired inequality (\ref{2.3}). This completes the proof of Lemma \ref{lemma1}.
\end{proof}

\vskip .3in
\section{Proof of Theorem \ref{t1}}
\label{s3}
In this section, we will prove  Theorem \ref{t1} by split the proof into three cases: $\alpha\ge2$, $\frac{3}{2}<\alpha<2$ and $1<\alpha\le\frac{3}{2}.$ Using the regularity criteria in Theorem \ref{A.1} in the  Appendix can easily follow the first case. Then via using standard energy method directly,  some simple manipulation yields  the second case. For the last case, we need exploiting the new structure of equations and tedious manipulation  involving a new type commutator estimate, see (\ref{3.12}). In the following proof, we denote $L^p([0,t],X)$ by $L^p_t(X)$ for some function spaces $X$.
\vskip .1in
Now, we begin the proof. The local well-posedness can be obtained in many ways, such as following the proof in the Chapter 3 of \cite{MB},  we omit the details. So there exists a $T_0>0$, such that (\ref{1.1}) has a unique solution $(u,\tau)$ satisfying
$$(u,\tau)\in C([0,T_0);H^s(\mathbb{R}^2)), \ u\in L^2([0,T_0); H^\alpha(\mathbb{R}^2)).$$
Thanks to the regularity criteria in Theorem \ref{A.1}, it suffices to show, $\forall\ T\ge 0,$
\begin{equation}\label{3.1}
\int_0^T \|\nabla u\|_{B_{\infty,\infty}^0}^2+\|\nabla \tau\|_{B_{\infty,\infty}^{-\alpha}}^2 dt<\infty.
\end{equation}
\vskip .1in
\underline{{\bf Case 1. $\alpha\ge 2$}} \ We obtain firstly the energy estimate of $u$ and $\tau$,
\begin{equation}\label{3.2}
\frac{1}{2}\frac{d}{dt}(\|u(t)\|_{L^2}^2+\|\tau(t)\|_{L^2}^2)+\nu\|\Lambda^\alpha u(t)\|_{L^2}^2=0
\end{equation}
or
\begin{equation}\label{3.3}
\sup_{0\le t\le T}(\|u(t)\|_{L^2}^2+\|\tau(t)\|_{L^2}^2)+\nu\int_0^T\|\Lambda^\alpha u(t)\|_{L^2}^2dt=\|u_0\|_{L^2}^2+\|\tau_0\|_{L^2}^2,
\end{equation}
which implies $u\in L^2_T(H^\alpha).$  By Bernstein's inequality,
$$\int_0^T\|\nabla u\|_{B_{\infty,\infty}^0}^2dt\le \int_0^T\|u\|_{H^\alpha}^2 dt<\infty$$
and
$$\int_0^T\|\nabla \tau\|_{B_{\infty,\infty}^{-\alpha}}^2 dt\le C\int_0^T (\|\tau\|_{L^2}^2+\|\tau\|_{B_{2,\infty}^{2-\alpha}}^2)dt\le C\int_0^T \|\tau\|_{L^2}^2 dt<\infty, $$
which yields (\ref{3.1}).
\vskip .1in
\underline{{\bf Case 2. $\frac{3}{2}<\alpha<2$}}\ \ Although the proof of this case is a little more tedious than the first case, we can finish it by using some elementary inequalities. Let $s_1\in (2-\alpha,\alpha-1)$, with a standard procedure follows
\begin{equation}\label{3.4}
\begin{aligned}
\frac{1}{2}\frac{d}{dt}\|u\|_{\tilde{H}^{s_1}}^2&+\nu\|\Lambda^\alpha u(t)\|_{\tilde{H}^{s_1}}^2=-\sum_{j\ge 0}2^{2js_1}([\Delta_j,u\cdot\nabla]u|\Delta_j u)
\\
&+\sum_{j\ge 0}2^{2js_1}(\Delta_j{\rm div}\tau|\Delta_j u):=J_{11}+J_{12}.
\end{aligned}
\end{equation}
Using (\ref{2.1}) and Young's inequality, we have
$$|J_{11}|\le C\|\nabla u\|_{B_{\infty,\infty}^{-\alpha}}\|u\|_{H^{s_1}}\|u\|_{H^{s_1+\alpha}}\le C(\|u\|_{H^\alpha}^2+1)\|u\|_{H^{s_1}}^2+\frac{\nu}{4}\|\Lambda^\alpha u\|_{H^{s_1}}^2.$$
Using (\ref{3.3}) and interpolation inequality,  for some $\theta\in (0,1),$
\begin{equation*}
\begin{aligned}
|J_{12}|\le& \|\tau\|_{L^2}\|\Lambda^{2s_1+1}u\|_{L^2}\le C\|u\|_{\dot{H}^{s_1}}^\theta\|u\|_{\dot{H}^{s_1+\alpha}}^{1-\theta}\\
\le& C(\|u\|_{\dot{H}^{s_1}}^2+1)+\frac{\nu}{4}\|\Lambda^\alpha u\|_{H^{s_1}}^2.
\end{aligned}
\end{equation*}
Inserting  the bound of $J_{11}$, $J_{12}$ into (\ref{3.4}), then combining with (\ref{3.2}) follows
\begin{equation}\label{3.5}
\frac{d}{dt}(\|u\|_{H^{s_1}}^2+\|\tau\|_{L^2}^2)+\nu\|\Lambda^\alpha u\|_{H^{s_1}}^2\le C(\|u\|_{H^\alpha}^2+1)\|u\|_{H^{s_1}}^2+C,
\end{equation}
which indicates that
\begin{equation}\label{3.6}
u\in L^\infty_T(H^{s_1})\cap L^2_T(H^{s_1+\alpha})
\end{equation}
by applying Gronwall's inequality to (\ref{3.5}). Thus the regularity of $u$ have been improved. The following process devotes to improving  the regularity of $\tau$.
Let $s_2\in (0,s_1+\alpha-1]$, we have
\begin{equation}\label{3.7}
\begin{aligned}
\frac{1}{2}\frac{d}{dt}\|\tau\|_{\tilde{H}^{s_2}}^2=&-\sum_{j\ge0}2^{2js_2}([\Delta_j,u\cdot\nabla]\tau|\Delta_j \tau)+
\sum_{j\ge0}2^{2js_2}(\Delta_j Du|\Delta_j \tau)\\
:=&J_{21}+J_{22}.
\end{aligned}
\end{equation}
Thanks to (\ref{2.3}),
\begin{equation*}
\begin{aligned}
|J_{21}|\le& C\|\nabla u\|_{L^\infty}\|\tau\|_{H^{s_2}}^2+C\|\nabla \tau\|_{B_{\infty,\infty}^{s_2-s_1-\alpha}}\|\tau\|_{H^{s_2}}\|u\|_{H^{s_1+\alpha}}\\
\le& C\|u\|_{H^{s_1+\alpha}}\|\tau\|_{H^{s_2}}^2+\|\tau\|_{H^{s_2}}^2\|u\|_{H^{s_1+\alpha}}.
\end{aligned}
\end{equation*}
For the estimate of $J_{22}$, simple manipulations  derive
$$|J_{22}|\le C\|\tau\|_{H^{s_2}}^2+C\|u\|_{H^{s_1+\alpha}}^2.$$
It follows from substituting the estimates of $J_{21}$ and $J_{22}$ in  (\ref{3.7}) and combining with (\ref{3.2})  that
$$\frac{d}{dt}(\|\tau\|_{H^{s_2}}^2+\|u\|_{L^2}^2)\le C(\|u\|_{H^{s_1+\alpha}}+1)\|\tau\|_{H^{s_2}}^2+C\|u\|_{H^{s_1+\alpha}}^2.$$
Integrating in time $[0,T]$, then using   Gronwall's inequality and (\ref{3.6}) yields
\begin{equation}\label{3.8}
\tau\in L^\infty_T(H^{s_2}),\ \forall\ s_2\in(0,s_1+\alpha-1].
\end{equation}
 Therefore, the improved regularity of $u$ (\ref{3.6}) and $\tau$ (\ref{3.8}) follows (\ref{3.1}). In fact, one has
 $$\int_0^T \|\nabla u\|_{B_{\infty,\infty}^0}^2 dt\le C\|u\|_{L^2_T(H^{s_1+\alpha})}^2<\infty$$
and
$$\int_0^T \|\nabla \tau\|_{B_{\infty,\infty}^{-\alpha}}^2 dt\le C \int_{0}^{T}\|\tau\|_{H^{2-\alpha}}^2dt\le C\int_{0}^{T}\|\tau\|_{H^{s_1+\alpha-1}}^2 dt<\infty.$$

\underline{{\bf Case 3. $1<\alpha\le \frac{3}{2}$}} We can get the vorticity equation by applying the operator curl to the first equation of (\ref{1.1}),
\begin{equation}\label{3.9}
\partial_t \omega+u\cdot\nabla \omega+\nu\Lambda^{2\alpha} \omega={\rm curldiv}\ \tau.
\end{equation}
Using the definition of $\mathcal{R}_\alpha$ in the section \ref{s1}, and denote $\Gamma:=\frac{1}{\nu}(\omega-\mathcal{R}_\alpha \tau)$. Then (\ref{3.9}) can be rewritten as
\begin{equation}\label{3.10}
\partial_t \omega+u\cdot\nabla \omega+\nu\Lambda^{2\alpha}\Gamma=0.
\end{equation}
Applying the operator $-\mathcal{R}_\alpha$ to the second equation of (\ref{1.1}), and adding the resulting equation to (\ref{3.10}), we have
$$\partial_t \Gamma+u\cdot\nabla \Gamma+\nu\Lambda^{2\alpha} \Gamma=[\mathcal{R}_\alpha,u\cdot\nabla]\tau+\frac{1}{2\nu}\Lambda^{2-2\alpha}\omega.$$
Taking the $L^2$ inner product with $\Gamma$ yields
\begin{equation}\label{3.11}
\begin{aligned}
\frac{1}{2}\frac{d}{dt}\|\Gamma\|_{L^2}^2+\nu\|\Lambda^\alpha \Gamma\|_{L^2}^2
\le& \frac{1}{2\nu}|(\Lambda^{2-2\alpha} \omega| \Gamma)|
+|([\mathcal{R}_\alpha,u\cdot\nabla]\tau|\Gamma)|\\
:=& K_1+K_2.
\end{aligned}
\end{equation}
By H\"{o}lder's inequality, interpolation inequality and Young's inequality,
\begin{equation*}
\begin{aligned}
|K_1|\le C\|\Lambda^{3-2\alpha}u\|_{L^2}\|\Gamma\|_{L^2}\le C\|u\|_{H^\alpha}\|\Gamma\|_{L^2}\le C\|u\|_{H^\alpha}^2+C\|\Gamma\|_{L^2}^2.
\end{aligned}
\end{equation*}
Using the following estimate,
\begin{equation}\label{3.12}
\|[\mathcal{R}_\alpha,u\cdot\nabla]\tau\|_{\dot{H}^{2\alpha-3}}^2\le C\|u\|_{H^\alpha}^2\|\tau\|_{L^2}^2,
\end{equation}
whose proof can be seen in the Appendix,  $K_2$ can be bounded as follows:
$$|K_2|\le C\|u\|_{H^\alpha}^2\|\tau\|_{L^2}^2+C\|\Lambda^{3-2\alpha} \Gamma\|_{L^2}^2\le C\|u\|_{H^\alpha}^2\|\tau\|_{L^2}^2+C\|\Gamma\|_{L^2}^2+\frac{\nu}{2}\|\Lambda^\alpha \Gamma\|_{L^2}^2.$$
Combining the bound of $K_1$ and $K_2$ in (\ref{3.11}), we have
$$\frac{1}{2}\frac{d}{dt}\|\Gamma\|_{L^2}^2+\nu\|\Lambda^\alpha \Gamma\|_{L^2}^2\le C\|u\|_{H^\alpha}^2\|\tau\|_{L^2}^2+C\| \Gamma\|_{L^2}^2+\frac{\nu}{2}\|\Lambda^\alpha \Gamma\|_{L^2}^2.$$
Absorbing the third term on the right hand side by the left hand side in the above inequality, then integrating the resulting inequality in $[0,T]$, we get
$$\sup_{0\le t\le T}\|\Gamma(t)\|_{L^2}^2+\nu\int_0^T \|\Lambda^\alpha \Gamma(t)\|_{L^2}^2 dt\le C(T,\nu, \|(u_0,\tau_0)\|_{L^2}, \|\Gamma_0\|_{L^2}),$$
that is
\begin{equation}\label{3.13}
\begin{aligned}
\sup_{0\le t\le T}\|(\omega-\mathcal{R}_\alpha \tau)(t)\|_{L^2}^2&+\nu\int_0^T \|\Lambda^\alpha(\omega-\mathcal{R}_\alpha \tau)(t)\|_{L^2}^2 dt\\
\le&
C(T,\nu, \alpha, \|(u_0,\tau_0)\|_{L^2}, \|\omega_0-\mathcal{R}_\alpha \tau_0\|_{L^2}).
\end{aligned}
\end{equation}
From the second equation in (\ref{1.1}), we get  the $L^q$ ($2\le q<\infty$) estimate of $\tau$:
\begin{equation}\label{3.131}
\frac{d}{dt}\|\tau\|_{L^q}\le C\|\omega\|_{L^q}.
\end{equation}
When $\alpha=\frac{3}{2},$  choosing $q=4$ in (\ref{3.131}) and using Sobolev's inequality $\|\omega\|_{L^4}\le C\|\Lambda^\frac{3}{2} u\|_{L^2}$,  we  get
\begin{equation}\label{3.14}
\tau\in L^\infty_T(L^4).
\end{equation}
When $1<\alpha<\frac{3}{2}$, let $a=\min\{\frac{3-2\alpha}{2\alpha-2},\frac{1}{2}\}$, choosing $\frac{1}{q}=a(\alpha-1)$ in (\ref{3.131}), with
$$\|\mathcal{R}_\alpha \tau\|_{L^q}\le C\|\tau\|_{L^\frac{1}{(a+1)(\alpha-1)}}$$
by Hardy-Littlewood-Sobolev Theorem ( see Chapter 5 of \cite{Stein}). We can also observe that $2\le \frac{1}{(a+1)(\alpha-1)}\le q,$ so that
\begin{equation*}
\begin{aligned}
\frac{d}{dt}\|\tau\|_{L^q}\le& C\|\omega-\mathcal{R}_\alpha\tau\|_{L^q}+C\|\mathcal{R}_\alpha \tau\|_{L^q}\\
\le& C\|\omega-\mathcal{R}_\alpha\tau\|_{L^q}+C\|\tau\|_{L^\frac{1}{(a+1)(\alpha-1)}}\\
\le& C\|\omega-\mathcal{R}_\alpha\tau\|_{L^q}+C\|\tau\|_{L^2}+C\|\tau\|_{L^q}.
\end{aligned}
\end{equation*}
Integrating in time $[0,T]$, thanks to (\ref{3.13}) which ensures $\omega-\mathcal{R}_\alpha \tau \in L^1_T(L^q)$ and using Gronwall's inequality, one gets
\begin{equation}\label{3.15}
\tau \in L^\infty_T(L^\frac{1}{a(\alpha-1)}),\   {\rm where}\  a=\min\{\frac{3-2\alpha}{2\alpha-2},\frac{1}{2}\}.
\end{equation}
Combining with (\ref{3.14}) and (\ref{3.15}) follows that
$$\int_0^T \|\nabla \tau\|_{B_{\infty,\infty}^{-\alpha}}^2 d\tau\le \int_0^T(\|\tau\|_{L^2}^2+\|\tau\|_{L^\frac{2}{\alpha-1}}^2)dt<\infty.$$
Using (\ref{3.13}), (\ref{3.14}) and (\ref{3.15}), and by interpolation inequality, we have
\begin{equation*}
\begin{aligned}
\int_0^T\|\omega\|_{\dot{B}_{\frac{1}{\alpha-1},\infty}^{2\alpha-2}}^2dt\le& \int_0^T\|\Lambda^{2\alpha-2}\omega\|_{L^\frac{1}{\alpha-1}}^2dt
\le C\int_0^T\|\Lambda^{2\alpha-2}(\omega-\mathcal{R}_\alpha \tau)\|_{L^\frac{1}{\alpha-1}}^2dt\\
&+
\int_0^T\|\Lambda^{2\alpha-2}\mathcal{R}_\alpha \tau\|_{L^\frac{1}{\alpha-1}}^2dt
<\infty,
\end{aligned}
\end{equation*}
which implies that
$$\int_0^T\|\omega\|_{B_{\infty,\infty}^0}^2 dt<\infty$$
by Bernstein's inequality. Hence,  we have proved (\ref{3.1}) and then concludes the proof of Theorem \ref{t1}.

\vskip .3in
\section{Proof of Theorem \ref{t2}}
\label{s4}
As the previous section,  we only give the global a priori estimates.
Thanks to the regularity criteria in Theorem \ref{A.1}, it suffices to show that $\forall\ T\ge 0,$
\begin{equation}\label{4.0}
\int_0^T \|\nabla u\|_{B_{\infty,\infty}^{-\beta}}^2+\|\nabla \tau\|_{B_{\infty,\infty}^{-1}}^2 dt<\infty.
\end{equation}
 We  need firstly a lemma.
\begin{lemma}\label{weisis}\cite{CRW}
Let $\mathcal{R}=(\mathcal{R}_1,\mathcal{R}_2,\cdot\cdot\cdot,\mathcal{R}_d)$ be the Riesz transform on $\R^d$. Then, the following commutator estimate holds
\begin{equation}\label{Weies}
\|[b,\mathcal{R}_i\mathcal{R}_j]f\|_{L^p}\le C(d,p)[b]_{BMO}\|f\|_{L^p},\ p\in (1,\infty),
\end{equation}
where the semi-norm $[b]_{BMO}$ is defined by
$$[b]_{BMO}:=\sup_{B}\frac{1}{|B|}\int_{B}|b-b_{B}|\  dx,\ b_{B}=\frac{1}{|B|}\int_{B} b(x) dx$$
and the supremum is taken over all balls in $\R^d$.
\end{lemma}
Now, we begin the proof with the energy estimate like (\ref{3.3}):
\begin{equation}\label{4.1}
\begin{aligned}
\sup_{0\le t\le T}(\|u(t)\|_{L^2}^2+\|\tau(t)\|_{L^2}^2)+&\nu\int_0^T\|\Lambda^\alpha u(t)\|_{L^2}^2 dt+\eta\int_0^T\|\Lambda^\beta \tau(t)\|_{L^2}^2dt\\
&=\|u_0\|_{L^2}^2+\|\tau_0\|_{L^2}^2.
\end{aligned}
\end{equation}
As the proof of the previous theorem for the case $1<\alpha\le\frac{3}{2}$,  we will exploit the structure of the equation. Similarly, we have the vorticity equations
\begin{equation}\label{4.2}
\partial_t \omega+u\cdot\nabla \omega-\nu\Delta\Gamma={\rm curldiv}\tau.
\end{equation}
Denote $\Gamma_1:=\omega-\mathcal{R}_1\tau$, here $\mathcal{R}_1$ is defined in the section \ref{s1}. Applying $-\mathcal{R}_1$ to the second equation of (\ref{1.1}), then adding the resulting equation to (\ref{4.2}) yields that
$$\partial_t \Gamma_1+u\cdot\nabla \Gamma_1-\nu\Delta \Gamma_1=[\mathcal{R}_1,u\cdot\nabla]\tau+\frac{1}{2\nu}\omega-\eta\Lambda^{2\beta}\mathcal{R}_1\tau.$$
Taking the $L^2$ inner product with $\Gamma_1$, then
\begin{equation}\label{4.3}
\begin{aligned}
\frac{1}{2}\frac{d}{dt}\|\Gamma_1\|_{L^2}^2+\nu\|\nabla \Gamma_1\|_{L^2}^2\le& |([\mathcal{R}_1,u\cdot\nabla]\tau|\Gamma_1)|
+\frac{1}{2\nu}|(\omega|\Gamma_1)|
+\eta|(\mathcal{R}_1\tau|\Lambda^{2\beta}\Gamma_1)|\\
:=&N_1+N_2+N_3.
\end{aligned}
\end{equation}
Thanks to the commutator estimate (\ref{Weies}), using ${\rm div}u=0$ and   integrating by parts , we have
\begin{equation*}
\begin{aligned}
|N_1|\le& \sum_{i=1}^3|([\mathcal{R}_1,u_i]\tau|\partial_i \Gamma_1)|\le \sum_{i=1}^3\|[\mathcal{R}_1,u_i]\tau\|_{L^2}\|\nabla \Gamma_1\|_{L^2}\\
\le& C[u]_{BMO}\|\tau\|_{L^2}\|\nabla \Gamma_1\|_{L^2}\le C\|\nabla u\|_{L^2}^2+\frac{\nu}{4}\|\nabla \Gamma_1\|_{L^2}^2.
\end{aligned}
\end{equation*}
The estimate of $N_2$ and $N_3$ can be easily obtained as follows:
$$|N_2|\le C\|\omega\|_{L^2}\|\Gamma_1\|_{L^2}\le C\|\omega\|_{L^2}^2+C\|\Gamma_1\|_{L^2}^2,$$
$$|N_3|\le C\|\Gamma_1\|_{L^2}^2+\frac{\nu}{4}\|\nabla \Gamma_1\|_{L^2}^2.$$
Inserting the above estimates into (\ref{4.3}) follows that
$$\frac{1}{2}\frac{d}{dt}\|\Gamma_1\|_{L^2}^2+\frac{\nu}{2}\|\nabla \Gamma_1\|_{L^2}^2\le C\|\nabla u\|_{L^2}^2+C\|\Gamma_1 \|_{L^2}^2,$$
which yields
$$\sup_{0\le t\le T}\|\Gamma_1(t)\|_{L^2}^2+
\frac{\nu}{2}\int_{0}^T\|\nabla \Gamma_1(t)\|_{L^2}^2dt\le C(T, \nu,\beta, \|(u_0,\tau_0)\|_{L^2},\|\Gamma_1(0)\|_{L^2})$$
by integrating in time, using energy estimate (\ref{4.1}) and applying the Gronwall's inequality. That is
\begin{equation}\label{4.4}
\begin{aligned}
\sup_{0\le t\le T}\|(\omega-\mathcal{R}_1\tau)(t)\|_{L^2}^2+&\int_0^T \|\nabla (\omega-\mathcal{R}_1\tau)(t)\|_{L^2}^2dt\\
\le&  C(T, \nu,\beta, \|(u_0,\tau_0)\|_{L^2},\|(\omega-\mathcal{R}_1\tau)(0)\|_{L^2}).
\end{aligned}
\end{equation}
We then deduce, by a  similar argument  as the previous theorem for the case $1<\alpha\le \frac{3}{2}$,
\begin{equation*}
\begin{aligned}
\frac{d}{dt}\|\tau\|_{L^p}\le& C\|\omega\|_{L^p}\le C\|\omega-\mathcal{R}_1\tau\|_{L^p}+C\|\mathcal{R}_1\tau\|_{L^p}\\
\le& C\|\omega-\mathcal{R}_1\tau\|_{L^p}+C\|\tau\|_{L^p},
\end{aligned}
\end{equation*}
which leads
$$\tau\in L^\infty_T(L^p),\ \forall\ 2\le p<\infty$$
by integrating in time, using (\ref{4.4}), interpolation inequality and Gronwall's inequality.
Using  (\ref{4.4}) and interpolation inequality again yields
\begin{equation}\label{4.5}
\begin{aligned}
\int_0^T \|\omega\|_{L^p}^2 dt\le C\int_0^T (\|\omega-\mathcal{R}_1\tau\|_{L^p}^2+\|\tau\|_{L^p}^2)dt<\infty.
\end{aligned}
\end{equation}
This implies
$$\int_0^T \|\nabla u\|_{B_{\infty,\infty}^{-\beta}}^2dt<\infty.$$
In fact, choosing $p=\frac{2}{\beta}$ in (\ref{4.5}), it follows from using the Bernstein's inequality that
\begin{equation*}
\begin{aligned}
\int_0^T \|\nabla u\|_{B_{\infty,\infty}^{-\beta}}^2dt\le& C\int_0^T \|\nabla u\|_{{B}_{\frac{2}{\beta},\infty}^0}^2 dt\\
\le& C\int_0^T \|u\|_{L^2}^2+\|\omega\|_{L^\frac{2}{\beta}}^2 dt<\infty.
\end{aligned}
\end{equation*}
To prove (\ref{4.0}), with energy estimate (\ref{4.1}), by Bernstein's inequality, it suffices to show
\begin{equation}\label{aim}
\int_{0}^T \|\tau\|_{\dot{B}_{q,\infty}^\frac{2}{q}}^2 dt\le C,\ {\rm for\ some}\ q>\frac{2}{\beta}.
\end{equation}
Following the standard argument,
\begin{equation}\label{4.6}
\frac{d}{dt}\|\dot{\Delta}_j\tau\|_{L^q}+c2^{2j\beta}\|\dot{\Delta}_j \tau\|_{L^q}\le \|[\dot{\Delta}_j,u\cdot\nabla]\tau\|_{L^q}
+C\|\Delta_j \omega\|_{L^q},
\end{equation}
where we have used the generalized Bernstein's inequality (see \cite{CMZ1}),
$$\int_{\R^2} \Lambda^{2\beta}\dot{\Delta}_j \tau \dot{\Delta}_j \tau |\dot{\Delta}_j \tau|^{q-2} dx\ge c2^{2j\beta}\|\dot{\Delta}_j \tau\|_{L^q}^q,\ \forall\ 2\le q<\infty.$$
Multiplying (\ref{4.6}) by $2^{j(\frac{2}{q}-\beta)}$ and taking the supremum over $j\in\Z$ on the both sides of the resulting inequality yields
\begin{equation}\label{40}
\frac{d}{dt}\|\tau\|_{\dot{B}_{q,\infty}^{\frac{2}{q}-\beta}}+c\|\tau\|_{\dot{B}_{q,\infty}^{\frac{2}{q}+\beta}}\le C\|\nabla u\|_{\dot{B}_{q,\infty}^{\frac{2}{q}-\beta}}+
\underbrace{\sup_{j\in \Z}2^{j(\frac{2}{q}-\beta)}\|[\dot{\Delta}_j,u\cdot\nabla]\tau\|_{L^q}}_{\Xi}.
\end{equation}
Next, we will bound $\Xi$. By the homogeneous Bony's decomposition,
\begin{equation*}
\begin{aligned}
|\Xi|\le &\sup_{j\in \Z}2^{j(\frac{2}{q}-\beta)}\{
\sum_{|k-j|\le4}\|[\dot{\Delta}_j, \dot{S}_{k-1}u\cdot\nabla]\dot{\Delta}_k\tau\|_{L^q}
+\sum_{|k-j|\le4}\|\dot{\Delta}_j(\dot{\Delta}_k u\cdot\nabla \dot{S}_{k-1}\tau)\|_{L^q}\\
&
+\sum_{k\ge j-2}\|\dot{\Delta}_k u\cdot\nabla \dot{\Delta}_j\dot{S}_{k+2}\tau\|_{L^q}
+\sum_{k\ge j-3}\|\dot{\Delta}_j(\dot{\Delta}_ku\cdot\nabla\tilde{\dot{\Delta}}_k\tau)\|_{L^q}
\}\\
=:&\Xi_1+\Xi_2+\Xi_3+\Xi_4.
\end{aligned}
\end{equation*}
By H\"{o}lder's inequality,  (\ref{SCE}) and (\ref{E}),
\begin{equation*}
\begin{aligned}
|\Xi_1|\le& \sup_{j\in \Z}2^{j(\frac{2}{q}-\beta)}\sum_{|k-j|\le4}\|\nabla \dot{S}_{k-1}u\|_{L^q}\|\dot{\Delta}_k\tau\|_{L^\infty}\\
\le& C\|\nabla u\|_{\dot{B}_{q,\infty}^{\frac{2}{q}-\beta}}\sup_{j\in \Z}\sum_{|k-j|\le4}2^{(j-k)(\frac{2}{q}-\beta)}\|\dot{\Delta}_k \tau\|_{L^\infty}\\
\le& C\|\nabla u\|_{\dot{B}_{q,\infty}^{\frac{2}{q}-\beta}}\|\tau\|_{\dot{B}_{\infty,\infty}^0}.
\end{aligned}
\end{equation*}
By H\"{o}lder's inequality and (\ref{E}),
\begin{equation*}
\begin{aligned}
|\Xi_2|\le& \sup_{j\in \Z}2^{j(\frac{2}{q}-\beta)}\sum_{|k-j|\le4}\|\dot{\Delta}_k u\|_{L^q}\|\nabla\dot{S}_{k-1}\tau\|_{L^\infty}\\
\le& C\|\tau\|_{\dot{B}_{\infty,\infty}^0}\sup_{j\in \Z}2^{j(\frac{2}{q}-\beta)}\sum_{|k-j|\le4}2^k\|\dot{\Delta}_k u\|_{L^q}\\
\le& C\|\tau\|_{\dot{B}_{\infty,\infty}^0}\sup_{j\in \Z}\sum_{|k-j|\le4}2^{(j-k)(\frac{2}{q}-\beta)}2^{k(\frac{2}{q}-\beta)}\|\nabla \dot{\Delta}_k u\|_{L^q}\\
\le& C\|\nabla u\|_{\dot{B}_{q,\infty}^{\frac{2}{q}-\beta}}\|\tau\|_{\dot{B}_{\infty,\infty}^0}.
\end{aligned}
\end{equation*}
By H\"{o}lder's inequality,  Bernstein's inequality and Young's inequality for series,
\begin{equation*}
\begin{aligned}
|\Xi_3|\le& C\sup_{j\in \Z}2^{j(\frac{2}{q}-\beta+1)}\|\dot{\Delta}_j \tau\|_{L^\infty}\sum_{k\ge j-2}\|\dot{\Delta}_k u\|_{L^q}\\
\le& C\|\tau\|_{\dot{B}_{\infty,\infty}^0}\sup_{j\in \Z}\sum_{k\ge j-2}2^{(j-k)(\frac{2}{q}-\beta+1)}2^{k(\frac{2}{q}-\beta+1)}\|\dot{\Delta}_k u\|_{L^q}\\
\le& C\|\tau\|_{\dot{B}_{\infty,\infty}^0} \|2^{k(\frac{2}{q}-\beta+1)}{\bf 1}_{k\le 2}\|_{l^1(\Z)}\left\|2^{k(\frac{2}{q}-\beta+1)}\|\dot{\Delta}_k u\|_{L^q}\right\|_{l^\infty(\Z)}\\
\le& C\|\nabla u\|_{\dot{B}_{q,\infty}^{\frac{2}{q}-\beta}}\|\tau\|_{\dot{B}_{\infty,\infty}^0},
\end{aligned}
\end{equation*}

\begin{equation*}
\begin{aligned}
|\Xi_4|\le& C\sup_{j\in \Z}2^{j(\frac{2}{q}-\beta+1)}\sum_{k\ge j-3}\|\dot{\Delta}_k u\|_{L^q}\|\tilde{\dot{\Delta}}_k \tau\|_{L^\infty}\\
\le& C\|\tau\|_{\dot{B}_{\infty,\infty}^0}\sup_{j\in \Z}\sum_{k\ge j-3}2^{(j-k)(\frac{2}{q}-\beta+1)}2^{k(1+\frac{2}{q}-\beta)}\|\dot{\Delta}_k u\|_{L^q}\\
\le& C\|\tau\|_{\dot{B}_{\infty,\infty}^0}\|2^{k(1+\frac{2}{q}-\beta)}{\bf 1}_{k\le 3}\|_{l^1(\Z)}\left\|2^{k(1+\frac{2}{q}-\beta)}\|\dot{\Delta}_k u\|_{L^q}\right\|_{l^\infty(\Z)}\\
\le& C\|\nabla u\|_{\dot{B}_{q,\infty}^{\frac{2}{q}-\beta}}\|\tau\|_{\dot{B}_{\infty,\infty}^0}.
\end{aligned}
\end{equation*}
Plugging the above estimates in (\ref{40}), using Bernstein's inequality and  interpolation inequality
$$\|\tau\|_{\dot{B}_{\infty,\infty}^0}\le C\|\tau\|_{\dot{B}_{p,\infty}^{\frac{2}{p}-\beta}}^\frac{1}{2}\|\tau\|_{\dot{B}_{p,\infty}^{\frac{2}{p}+\beta}}^\frac{1}{2}$$
and Young's inequality yields
\begin{equation}\label{4.7}
\frac{d}{dt}\|\tau\|_{\dot{B}_{q,\infty}^{\frac{2}{q}-\beta}}+\frac{c}{2}\|\tau\|_{\dot{B}_{q,\infty}^{\frac{2}{q}+\beta}}\le
C\|\omega\|_{\dot{B}_{q,\infty}^{\frac{2}{q}-\beta}}+C\|\omega\|_{\dot{B}_{q,\infty}^{\frac{2}{q}-\beta}}^2\|\tau\|_{\dot{B}_{q,\infty}^{\frac{2}{q}-\beta}},
\end{equation}
where we have used the bound of the Riesz transforms in homogeneous Besov space.  Thanks to $\tau\in L^\infty_T(L^p)$, $\forall\ p\in[2,\infty)$, and (\ref{4.4}), by Bernstein's inequality and $q>\frac{2}{\beta}$,
\begin{equation}\label{4.8}
\begin{aligned}
\int_0^T\|\omega\|_{\dot{B}_{q,\infty}^{\frac{2}{q}-\beta}}^2 dt\le& C\int_0^T \|\omega-\mathcal{R}_1\tau\|_{\dot{B}_{q,\infty}^{\frac{2}{q}-\beta}}^2 dt
+C\int_0^T \|\tau\|_{\dot{B}_{q,\infty}^{\frac{2}{q}-\beta}}^2dt\\
\le &C\int_0^T \|\omega-\mathcal{R}_1\tau\|_{\dot{H}^{1-\beta}}^2 dt+C\int_0^T \|\tau\|_{\dot{B}_{\frac{2}{\beta},\infty}^0}^2dt\\
\le& C\int_0^T \|\omega-\mathcal{R}_1\tau\|_{\dot{H}^{1-\beta}}^2 dt+C\int_0^T \|\tau\|_{L^\frac{2}{\beta}}^2dt\\
<&\infty.
\end{aligned}
\end{equation}
After integrating (\ref{4.7}) in $[0,T]$,  using (\ref{4.8}) and Gronwall's inequality derives
$$\sup_{0\le t\le T}\|\tau(t)\|_{\dot{B}_{q,\infty}^{\frac{2}{q}-\beta}}+c\int_0^T\|\tau\|_{\dot{B}_{q,\infty}^{\frac{2}{q}+\beta}}dt<\infty,$$
which implies (\ref{aim}) by interpolation inequality,
$$\|\tau\|_{L^2_T(\dot{B}_{q,\infty}^\frac{2}{q})}\le C\|\tau\|_{L^\infty_T(\dot{B}_{q,\infty}^{\frac{2}{q}-\beta})}^\frac{1}{2}\|\tau\|_{L^1_T(\dot{B}_{q,\infty}^{\frac{2}{q}+\beta})}^\frac{1}{2}.$$
This completes the proof of Theorem \ref{t2}.


\vskip .4in

\appendix
\section{}
\label{s6}
In this section, we will prove the regularity criteria for (\ref{1.1}) in general cases based on Littlewood-Palay Theory, and then give the proof of (\ref{3.12}) which plays the important role in the proof of Theorem \ref{t1}.
\begin{thm}\label{A.1}
Consider (\ref{1.1}) with $\alpha>0$ and the initial data $(u_0,\tau_0)\in H^s(\R^2)$, $s>2$. If $(u,\tau)$ satisfies the following condition:
$$\int_0^T \|\nabla u\|_{B_{\infty,\infty}^0}^2+\|\nabla \tau\|_{B_{\infty,\infty}^{-\alpha}}^2 dt<\infty,\ \ {\rm if}\ \ \eta=0,$$
or
$$\int_0^T \|\nabla u\|_{B_{\infty,\infty}^{-\min\{\alpha,\beta\}}}^2+\|\nabla \tau\|_{B_{\infty,\infty}^{-\alpha}}^2 dt<\infty,\ \ {\rm if}\ \eta>0,\beta>0,$$
then $(u,\tau)$ remains regular in $[0,T]$.
\end{thm}
\begin{proof}
It suffices to give the global a priori bound. By a standard process, we have
\begin{equation*}\label{A.2}
\begin{aligned}
\frac{1}{2}\frac{d}{dt}(\|u\|_{\tilde{H}^s}^2&+\|\tau\|_{\tilde{H}^s}^2)+\nu\|\Lambda^\alpha u\|_{\tilde{H}^s}^2+\eta\|\Lambda^\alpha u\|_{\tilde{H}^s}^2\\
=&-\sum_{j\ge 0}2^{2js}([\Delta_j,u\cdot\nabla]u|\Delta_j u)-\sum_{j\ge 0}2^{2js}([\Delta_j,u\cdot\nabla]\tau|\Delta_j \tau)\\
:=& I_1+I_2,
\end{aligned}
\end{equation*}
where the following cancelation property have been used
$$(u\cdot\nabla\Delta_j u|\Delta_j u)=(u\cdot\nabla\Delta_j \tau|\Delta_j \tau)=0.$$
Combining with the energy estimate of $u$ and $\tau$:
$$\frac{1}{2}\frac{d}{dt}(\|u\|_{L^2}^2+\|\tau\|_{L^2}^2)+\nu\|\Lambda^\alpha u\|_{L^2}^2+\eta\|\Lambda^\beta \tau\|_{L^2}^2=0,$$
we derive
\begin{equation}\label{A.2}
\frac{1}{2}\frac{d}{dt}(\|u\|_{H^s}^2+\|\tau\|_{H^s}^2)+\nu\|\Lambda^\alpha u\|_{H^s}^2+\eta\|\Lambda^\beta \tau\|_{H^s}^2=I_1+I_2.
\end{equation}
For the estimate of $I_1$, using (\ref{2.1}), we have
\begin{equation}\label{5.1}
|I_1|\le C\|\nabla u\|_{B_{\infty,\infty}^{-\alpha}}\|u\|_{H^s}\|u\|_{H^{s+\alpha}}\le C(\|\nabla u\|_{B_{\infty,\infty}^{-\alpha}}^2+1)
\|u\|_{H^s}^2+\frac{\nu}{4}\|\Lambda^\alpha u\|_{H^s}^2.
\end{equation}
The estimate of $I_2$ is split into two different situations  depend on the condition of $\beta$ and $\eta$.
\vskip .1in
\underline{$\alpha>0,\eta>0,\beta>0$.}\ Applying (\ref{2.2}), we have
\begin{equation}\label{5.11}
\begin{aligned}
|I_2|\le& C\|\nabla u\|_{B_{\infty,\infty}^{-\beta}}\|\tau\|_{H^s}\|\tau\|_{H^{s+\beta}}+C\|\nabla \tau\|_{B_{\infty,\infty}^{-\alpha}}\|\tau\|_{H^s}\|u\|_{H^{s+\alpha}}\\
\le& C(\|\nabla u\|_{B_{\infty,\infty}^{-\beta}}^2+\|\nabla \tau\|_{B_{\infty,\infty}^{-\alpha}}^2+1)(\|u\|_{H^s}^2+\|\tau\|_{H^s}^2)\\
&
+\frac{\nu}{4}\|\Lambda^\alpha u\|_{H^s}^2+\frac{\eta}{4}\|\Lambda^\beta \tau\|_{H^s}^2.
\end{aligned}
\end{equation}
Combining with (\ref{5.1}) and (\ref{5.11}) in (\ref{A.2}), and applying the Gronwall's inequality yields the desired result.
\vskip .1in
\underline{$\alpha>0,\eta=0$.} \ Applying (\ref{2.3}) with $s_1=s_2=s$,
thanks to the Log-interpolation inequality,
$$\|\nabla u\|_{L^\infty}\le C(\|\nabla u\|_{B_{\infty,\infty}^0}+1)\log (e+\|u\|_{H^s}),\ \ s>2,$$
we obtain
\begin{equation}\label{5.2}
\begin{aligned}
|I_2|\le& C\|\nabla u\|_{L^\infty}\|\tau\|_{H^s}^2+C\|\nabla \tau\|_{B_{\infty,\infty}^{-\alpha}}\|\tau\|_{H^s}\|u\|_{H^{s+\alpha}}\\
\le& C(\|\nabla u\|_{B_{\infty,\infty}^0}^2+\|\nabla \tau\|_{B_{\infty,\infty}^{-\alpha}}^2+1)\log(e+\|u\|_{H^s}^2+\|\tau\|_{H^s}^2)\\
&\times(\|u\|_{H^s}^2+\|\tau\|_{H^s}^2) +\frac{\nu}{4}\|\Lambda^\alpha u\|_{H^s}^2,
\end{aligned}
\end{equation}
which yields the desired result by combining with (\ref{5.1}) and using the Gronwall's inequality.
\end{proof}

\begin{proof}[Proof of (\ref{3.12})] The proof is based on homogeneous Bony's decomposition by modifying  the proof in \cite{Wan} (see Proposition 2.7 there).
Since $\|f\|_{\dot{H}^{2\alpha-3}}\thickapprox \|f\|_{\dot{B}_{2,2}^{2\alpha-3}}$, we have
\begin{equation}\label{5.3}
\begin{aligned}
\|[\mathcal{R}_\alpha,&u\cdot\nabla]\tau\|_{\dot{H}^{2\alpha-3}}^2\le C\sum_{j\in\Z}2^{2j(2\alpha-3)}\|\dot{\Delta}_j[\mathcal{R}_\alpha,u\cdot\nabla]\tau\|_{L^2}^2\\
&\le C\sum_{j\in\Z}2^{2j(2\alpha-3)}\|[\dot{\Delta}_j,u\cdot\nabla]\mathcal{R}_\alpha \tau\|_{L^2}^2
+C\sum_{j\in\Z}2^{2j(2\alpha-3)}\|[\dot{\Delta}_j\mathcal{R}_\alpha,u\cdot\nabla]\tau\|_{L^2}^2\\
&:=M_1+M_2.
\end{aligned}
\end{equation}
By homogeneous Bony's decomposition, $\Theta_{1}=[\dot{\Delta}_j, u \cdot \nabla]\mathcal{R}_\alpha\tau,$
\begin{equation*}
\begin{aligned}
\Theta_1=&\sum_{|k-j|\le4}[\dot{\Delta}_j,\dot{S}_{k-1}u\cdot\nabla]\dot{\Delta}_k\mathcal{R}_\alpha\tau
+\sum_{|k-j|\le4}\dot{\Delta}_j(\dot{\Delta}_ku\cdot\nabla\dot{S}_{k-1}\mathcal{R}_\alpha\tau)\\
&+\sum_{k\ge j-2}\dot{\Delta}_ku\cdot\nabla\dot{\Delta}_j\dot{S}_{k+2}\mathcal{R}_\alpha\tau
+\sum_{k\ge j-3}\dot{\Delta}_j(\dot{\Delta}_ku\cdot\nabla\tilde{\dot{\Delta}}_k\mathcal{R}_\alpha\tau)\\
:=& M_{11}+M_{12}+M_{13}+M_{14}.
\end{aligned}
\end{equation*}
By H\"{o}lder's inequality, (\ref{SCE}) and Bernstein's inequality,
\begin{equation*}
\begin{aligned}
\|M_{11}\|_{L^2}\le& \sum_{|k-j|\le4}\|\nabla\dot{S}_{k-1}u\|_{L^\infty}\|\dot{\Delta}_k\mathcal{R}_\alpha\tau\|_{L^2}\le C\|u\|_{L^\infty}\sum_{|k-j|\le4}2^k\|\dot{\Delta}_k\mathcal{R}_\alpha\tau\|_{L^2}\\
\le& C2^{j(3-2\alpha)}\|u\|_{L^\infty}\sum_{|k-j|\le4}2^{(j-k)(2\alpha-3)}2^{k(2\alpha-2)}\|\dot{\Delta}_k \mathcal{R}_\alpha\tau\|_{L^2}\\
\le& C2^{j(3-2\alpha)}\|u\|_{L^\infty}\|\mathcal{R}_\alpha \tau\|_{\dot{H}^{2\alpha-2}}\sum_{|k-j|\le4}2^{(j-k)(2\alpha-3)}c_k\\
\le& C2^{j(3-2\alpha)}c_j\|u\|_{L^\infty}\|\tau\|_{L^2}.
\end{aligned}
\end{equation*}
By H\"{o}lder's inequality, Bernstein's inequality and (\ref{E}),
\begin{equation*}
\begin{aligned}
\|M_{12}\|_{L^2}\le& \sum_{|k-j|\le4}\|\nabla\dot{S}_{k-1}\mathcal{R}_\alpha u\|_{L^\infty}\|\dot{\Delta}_k u\|_{L^2}\\
\le&C\sum_{|k-j|\le4}2^{-k}\|\nabla\dot{S}_{k-1}\mathcal{R}_\alpha u\|_{L^\infty}\|\nabla\dot{\Delta}_k u\|_{L^2}\\
\le& C\|\nabla u\|_{L^2}\sum_{|k-j|\le4}2^{-k}\|\nabla\dot{S}_{k-1}\mathcal{R}_\alpha u\|_{L^\infty}\\
\le& C2^{j(3-2\alpha)}\|\nabla u\|_{L^2}\sum_{|k-j|\le4}2^{(j-k)(2\alpha-3)}2^{k(2\alpha-4)}\|\nabla\dot{S}_{k-1}\mathcal{R}_\alpha\tau\|_{L^\infty}\\
\le& C2^{j(3-2\alpha)}\|\nabla u\|_{L^2}\|\nabla\mathcal{R}_\alpha\tau\|_{\dot{B}_{\infty,2}^{2\alpha-4}}\sum_{|k-j|\le4}2^{(j-k)(2\alpha-3)}c_k\\
\le& C2^{j(3-2\alpha)}c_j\|\nabla u\|_{L^2}\|\tau\|_{L^2}.
\end{aligned}
\end{equation*}
Similarly,
\begin{equation*}
\begin{aligned}
\|M_{13}\|_{L^2}\le& C2^j\|\dot{\Delta}_j\mathcal{R}_\alpha\tau\|_{L^\infty}\sum_{k\ge j-2}\|\dot{\Delta}_k u\|_{L^2}\\
\le& C 2^j\|\dot{\Delta}_j\mathcal{R}_\alpha\tau\|_{L^\infty}\sum_{k\ge j-2}2^{-k}\|\nabla\dot{\Delta}_k u\|_{L^2}\\
\le& C\|\nabla u\|_{L^2}\|\dot{\Delta}_j\mathcal{R}_\alpha\tau\|_{L^\infty}\\
\le& C2^{j(3-2\alpha)}\|\nabla u\|_{L^2}2^{j(2\alpha-3)}\|\dot{\Delta}_j\mathcal{R}_\alpha\tau\|_{L^\infty}\\
\le& C2^{j(3-2\alpha)}c_j\|\nabla u\|_{L^2}\|\mathcal{R}_\alpha\tau\|_{\dot{B}_{\infty,2}^{{2\alpha-3}}}\\
\le& C2^{j(3-2\alpha)}c_j\|\nabla u\|_{L^2}\|\tau\|_{L^2},
\end{aligned}
\end{equation*}
with the application of  Young's inequality for series,
\begin{equation*}
\begin{aligned}
\|M_{14}\|_{L^2}\le& C2^{2j}\sum_{k\ge j-3}\|\dot{\Delta}_k u\|_{L^2}\|\tilde{\dot{\Delta}}_k\mathcal{R}_\alpha\tau\|_{L^2}\\
\le& C\|\nabla u\|_{L^2}2^{2j}\sum_{k\ge j-3}2^{-k}\|\tilde{\dot{\Delta}}_k\mathcal{R}_\alpha\tau\|_{L^2}\\
\le& C2^{j(3-2\alpha)}\|\nabla u\|_{L^2}\sum_{k\ge j-3}2^{(j-k)(2\alpha-1)}2^{k(2\alpha-2)}\|\tilde{\dot{\Delta}}_k\mathcal{R}_\alpha \tau\|_{L^2}\\
\le& C2^{j(3-2\alpha)}\|\nabla u\|_{L^2}\|\mathcal{R}_\alpha \tau\|_{\dot{H}^{2\alpha-2}}\sum_{k\ge j-3}2^{(j-k)(2\alpha-1)}c_k\\
\le& C2^{j(3-2\alpha)}c_j\|\nabla u\|_{L^2}\|\tau\|_{L^2}.
\end{aligned}
\end{equation*}
Thus we have
$$|M_1|\le C(\|\nabla u\|_{L^2}^2+\|u\|_{L^\infty}^2)\|\tau\|_{L^2}^2\le C\|u\|_{H^\alpha}^2\|\tau\|_{L^2}^2.$$
Next, we bound $M_2$. Using the homogeneous Bony's decomposition again, let $\Theta_2=[\dot{\Delta}_j\mathcal{R}_\alpha,u\cdot\nabla]\tau,$
\begin{equation*}
\begin{aligned}
\Theta_2=&\sum_{|k-j|\le4}[\dot{\Delta}_j\mathcal{R}_\alpha, \dot{S}_{k-1}u\cdot\nabla]\dot{\Delta}_k\tau
+\sum_{|k-j|\le4}\dot{\Delta}_j\mathcal{R}_\alpha(\dot{\Delta}_ku\cdot\nabla\dot{S}_{k-1}\tau)\\
&+\sum_{k\ge j-2}\dot{\Delta}_ku\cdot\nabla\dot{\Delta}_j\mathcal{R}_\alpha\dot{S}_{k+2}\tau
+\sum_{k\ge j-3}\dot{\Delta}_j\mathcal{R}_\alpha(\dot{\Delta}_ku\cdot\nabla \tilde{\dot{\Delta}}_k\tau)\\
:=& M_{21}+M_{22}+M_{23}+M_{24}.
\end{aligned}
\end{equation*}
Since
$$\widehat{\dot{\Delta}_j\mathcal{R}_\alpha f}=\frac{1}{\nu}\varphi(2^{-j}\xi)\frac{\xi_i\xi_j}{|\xi|^{2\alpha}}\widehat{f},$$
one has
$$\dot{\Delta}_j\mathcal{R}_\alpha f=2^{j(4-2\alpha)}\mathfrak{h}(2^j\cdot)\star f,\ \ \rm {for \ some} \ \mathfrak{h}\in \mathcal{S}.$$
Thus using (\ref{WuJMFM}) with $p_1=1$ and $p_2=2$,
\begin{equation*}
\begin{aligned}
\|M_{21}\|_{L^2}\le& C2^{j(3-2\alpha)}\|2^jx\mathfrak{h}(2^j x)\|_{L^1}\sum_{|k-j|\le4}\|\nabla \dot{S}_{k-1}u\|_{L^\infty}\|\nabla \dot{\Delta}_k\tau\|_{L^2}\\
\le& C2^{j(3-2\alpha)}\|u\|_{L^\infty}\sum_{|k-j|\le4}\|\Delta_k \tau\|_{L^2}\\
\le& C2^{j(3-2\alpha)}c_j\|u\|_{L^\infty}\|\tau\|_{L^2}.
\end{aligned}
\end{equation*}
By Bernstein's inequality and H\"{o}lder's inequality,
\begin{equation*}
\begin{aligned}
\|M_{22}\|_{L^2}\le& C2^{j(2-2\alpha)}\sum_{|k-j|\le4}\|\dot{\Delta}_k u\|_{L^\infty}\|\nabla \dot{S}_{k-1}\tau\|_{L^2}\\
\le& C2^{j(3-2\alpha)}\|\tau\|_{L^2}\sum_{|k-j|\le4}2^{k-j}\|\dot{\Delta}_k u\|_{L^\infty}\\
\le& C2^{j(3-2\alpha)}\|\tau\|_{L^2}\sum_{|k-j|\le4}2^{k-j}2^k\|\dot{\Delta}_k u\|_{L^2}\\
\le& C2^{j(3-2\alpha)}\|\tau\|_{L^2}\|\nabla u\|_{L^2}\sum_{|k-j|\le 4}2^{k-j}c_k\\
\le& C2^{j(3-2\alpha)}c_j\|\tau\|_{L^2}\|\nabla u\|_{L^2}.
\end{aligned}
\end{equation*}
Similarly,
\begin{equation*}
\begin{aligned}
\|M_{23}\|_{L^2}\le& C2^j\|\dot{\Delta}_j \mathcal{R}_\alpha\tau\|_{L^\infty}\sum_{k\ge j-2}\|\dot{\Delta}_k u\|_{L^2}\\
\le& C2^j\|\dot{\Delta}_j \mathcal{R}_\alpha\tau\|_{L^\infty}\sum_{k\ge j-2}2^{-k}\|\nabla\dot{\Delta}_k u\|_{L^2}\\
\le& C2^{j(3-2\alpha)}\|\nabla u\|_{L^2}2^{j(2\alpha-3)}\|\dot{\Delta}_j \mathcal{R}_\alpha\tau\|_{L^\infty}\\
\le& C2^{j(3-2\alpha)}c_j\|\nabla u\|_{L^2}\|\mathcal{R}_\alpha \tau\|_{\dot{B}_{\infty,2}^{2\alpha-3}}\\
\le& C2^{j(3-2\alpha)}c_j\|\nabla u\|_{L^2}\|\tau\|_{L^2}.
\end{aligned}
\end{equation*}
Using Young's inequality for series, we get
\begin{equation*}
\begin{aligned}
\|M_{24}\|_{L^2}\le& C2^{j(4-2\alpha)}\sum_{k\ge j-3}\|\dot{\Delta}_k u\|_{L^2}\|\tilde{\dot{\Delta}}_k \tau\|_{L^2}\\
\le& C2^{j(4-2\alpha)}\|\tau\|_{L^2}\sum_{k\ge j-3}\|\dot{\Delta}_k u\|_{L^2}\\
\le& C2^{j(3-2\alpha)}\|\tau\|_{L^2}\sum_{k\ge j-3}2^{j-k}2^k\|\dot{\Delta}_k u\|_{L^2}\\
\le& C2^{j(3-2\alpha)}c_j\|\tau\|_{L^2}\|\nabla u\|_{L^2}.
\end{aligned}
\end{equation*}
As a consequence,
$$|M_2|\le C(\|\nabla u\|_{L^2}^2+\|u\|_{L^\infty}^2)\|\tau\|_{L^2}^2\le C\|u\|_{H^\alpha}^2\|\tau\|_{L^2}^2.$$
Together with the estimate of $M_1$ in (\ref{5.3}) can yield the desired inequality (\ref{3.12}).
\end{proof}
\vskip .4in
\section*{Acknowledgements}
Special thanks go to   Prof. Majdoub for kindly notifying the author of  his interesting work \cite{BM}.

\end{document}